\documentclass{article}

\usepackage[margin=4cm]{geometry}
\usepackage{amsmath,amsthm,amssymb,mathtools}
\usepackage[authoryear]{natbib}
\usepackage{paralist}
\usepackage[inline]{enumitem}
\setdefaultleftmargin{2em}{}{}{}{}{}

\usepackage{bbm}
\usepackage{mathrsfs}

\usepackage{textcase}

\usepackage[breaklinks=true]{hyperref}

\usepackage[displaymath]{lineno}
\numberwithin{equation}{section}
\allowdisplaybreaks[4]

\theoremstyle{plain} 
\newtheorem{theorem}{Theorem}[section]
\newtheorem{proposition}[theorem]{Proposition}
\newtheorem{lemma}[theorem]{Lemma}

\theoremstyle{definition}

\makeatletter

\renewcommand{\cite}{\citet}

\def\^#1{\relax\ifmmode {\mathaccent"705E #1} \else {\accent94 #1} \fi}
\def\~#1{\relax\ifmmode {\mathaccent"707E #1} \else {\accent"7E #1} \fi}

\def\*#1{\relax#1^\ast}
\edef\-#1{\relax\noexpand\ifmmode {\noexpand\bar{#1}} \noexpand\else \-#1\noexpand\fi}
\def\>#1{\vec{#1}}
\def\.#1{\dot{#1}}

\def\atop{\@@atop}
\def\%#1{\mathcal{#1}}

\renewcommand{\leq}{\leqslant}
\renewcommand{\geq}{\geqslant}
\renewcommand{\phi}{\varphi}
\newcommand{\eps}{\varepsilon}

\newcommand{\eq}{\eqref}

\newcommand{\dtv}{\mathop{d_{\mathrm{TV}}}}

\newcommand\indep{\protect\mathpalette{\protect\independenT}{\perp}}
\def\independenT#1#2{\mathrel{\rlap{$#1#2$}\mkern2mu{#1#2}}}

\newcommand{\toinf}{\to\infty}

\newcommand{\Po}{\mathop{\mathrm{Poisson}}}
\newcommand{\Vol}{\mathop{\mathrm{Vol}}}
\newcommand{\law}{\mathscr{L}}
\newcommand{\eqlaw}{\stackrel{\mathscr{D}}{=}}
\newcommand{\dbox}{d_\square}

\def\be#1{\begin{linenomath*}\begin{equation*}#1\end{equation*}\end{linenomath*}}
\def\ben#1{\begin{linenomath*}\begin{equation}#1\end{equation}\end{linenomath*}}
\def\bes#1{\begin{linenomath*}\begin{equation*}\begin{split}#1\end{split}\end{equation*}\end{linenomath*}}
\def\besn#1{\begin{linenomath*}\begin{equation}\begin{split}#1\end{split}\end{equation}\end{linenomath*}}

\def\bm#1{\begin{linenomath*}\begin{multline*}#1\end{multline*}\end{linenomath*}}
\def\bmn#1{\begin{linenomath*}\begin{multline}#1\end{multline}\end{linenomath*}}

\def\given{\mskip 0.5mu plus 0.25mu\vert\mskip 0.5mu plus 0.15mu}
\newcounter{@bracketlevel}
\def\@bracketfactory#1#2#3#4#5#6{
\expandafter\def\csname#1\endcsname##1{%
\addtocounter{@bracketlevel}{1}%
\global\expandafter\let\csname @middummy\alph{@bracketlevel}\endcsname\given%
\global\def\given{\mskip#5\csname#4\endcsname\vert\mskip#6}\csname#4l\endcsname#2##1\csname#4r\endcsname#3%
\global\expandafter\let\expandafter\given\csname @middummy\alph{@bracketlevel}\endcsname
\addtocounter{@bracketlevel}{-1}}%
}
\def\bracketfactory#1#2#3{%
\@bracketfactory{#1}{#2}{#3}{relax}{0.5mu plus 0.25mu}{0.5mu plus 0.15mu}
\@bracketfactory{b#1}{#2}{#3}{big}{1mu plus 0.25mu minus 0.25mu}{0.6mu plus 0.15mu minus 0.15mu}
\@bracketfactory{bb#1}{#2}{#3}{Big}{2.4mu plus 0.8mu minus 0.8mu}{1.8mu plus 0.6mu minus 0.6mu}
\@bracketfactory{bbb#1}{#2}{#3}{bigg}{3.2mu plus 1mu minus 1mu}{2.4mu plus 0.75mu minus 0.75mu}
\@bracketfactory{bbbb#1}{#2}{#3}{Bigg}{4mu plus 1mu minus 1mu}{3mu plus 0.75mu minus 0.75mu}
}
\bracketfactory{klg}{\lbrace}{\rbrace}
\bracketfactory{klr}{(}{)}
\bracketfactory{kle}{[}{]}
\bracketfactory{abs}{\lvert}{\rvert}
\bracketfactory{norm}{\Vert}{\Vert}
\bracketfactory{floor}{\lfloor}{\rfloor}
\bracketfactory{ceil}{\lceil}{\rceil}
\bracketfactory{angle}{\langle}{\rangle}

\newcounter{ctr}\loop\stepcounter{ctr}\edef\X{\@Alph\c@ctr}%
	\expandafter\edef\csname s\X\endcsname{\noexpand\mathscr{\X}}
	\expandafter\edef\csname c\X\endcsname{\noexpand\mathcal{\X}}
	\expandafter\edef\csname b\X\endcsname{\noexpand\boldsymbol{\X}}
	\expandafter\edef\csname I\X\endcsname{\noexpand\mathbbm{\X}}
\ifnum\thectr<26\repeat

\let\@IP\IP\let\IP\undefined
\DeclareMathOperator{\IP}{\@IP}
\DeclareMathOperator{\I}{\mathrm{I}}

\renewcommand\section{\@startsection {section}{1}{\z@}%
{-3.5ex \@plus -1ex \@minus -.2ex}%
{1.3ex \@plus.2ex}%
{\center\small\sc\MakeTextUppercase}}

\def\subsection#1{\@startsection {subsection}{2}{0pt}%
{-3.5ex \@plus -1ex \@minus -.2ex}%
{1ex \@plus.2ex}%
{\bf\mathversion{bold}}{#1}}

\def\subsubsection#1{\@startsection{subsubsection}{3}{0pt}%
{\medskipamount}%
{-10pt}%
{\normalsize\itshape}{\kern-2.2ex. #1.}}

\makeatother

\begin{document}

\title{\sc\bf\large\MakeUppercase{Respondent driven sampling and sparse graph convergence}}
\author{\sc Siva Athreya$^{1}$ \and \sc Adrian R\"ollin$^{2}$}
\date{}
\maketitle

\footnotetext[1]{Indian Statistical Institute, 8th Mile Mysore Road, Bangalore,
560059 India.\\ Email: athreya@isibang.ac.in}

\footnotetext[2]{Department of Statistics and Applied Probability, National University of
Singapore, 6~Science Drive~2, Singapore 117546. Email: adrian.roellin@nus.edu.sg } 

\begin{abstract} 
We consider a particular respondent-driven sampling procedure governed by a graphon. By a specific clumping procedure of the sampled vertices we construct a sequence of sparse graphs. If the sequence of the vertex-sets is stationary then the sequence of sparse graphs converge to the governing graphon in the cut-metric. The tools used are concentration inequality for Markov chains and the Stein-Chen method.\end{abstract}



\section{Introduction}

\noindent Respondent Driven Sampling (RDS), popularised by
\cite{Heckathorn1997}, is a method to sample from hard-to-reach
populations, such as drug users, MSM and people with HIV, and it is
being routinely used in studies involving such populations. The
sampling procedure is subject to various biases, one of which is a
bias towards individuals with higher degrees, as these are more likely
to appear in the sample.

How this bias affects the network as a whole has been described by
\cite{Athreya2016} in the context of dense graph limits. The model
considered there is defined in terms of a two-step procedure. First,
vertices are sampled according to an ergodic process (the important
point to note is that the vertices need not be sampled independently
of each other). Second, edges between vertices are sampled
independently of each other, where the probability of an edge is
determined via a \emph{graphon} representing the underlying network.

Dense graphs are at one extreme of graph sequences. These are graphs on~$n$ vertices with the number of edges being of order~$n^2$, which is far more than what is observed in real world networks. At the
opposite end are sequences of graphs with bounded (average) degree
and consequently having order~$n$ edges. These have a separate
limiting theory which is not quite applicable to many real world networks. There is class of graph sequences between these two extremes, called \emph{sparse graphs} --- these are graphs for which the average degree grows in the number of vertices, but only at sub-linear speed.

The purpose of this note is to extend the work of
\cite{Athreya2016} to sparse graphs, and to consider more realistic models of sampling. Since RDS data typically comes in the
form of trees, the actual graphs are those with average degrees
remaining bounded as the number of nodes~$n$ grows. We propose a model
where ``close enough'' participants are ``clumped'' together so that
the average degree now grows in~$n$. Our main result is that the
random sparse graph sequence obtained through a specific
respondent-driven sampling procedure converges almost surely to the
graphon underlying the network in the cut-metric, provided the
sequence of the vertex-sets is stationary

The method of proof in this article is entirely different from that of
\cite{Athreya2016}. This is mainly due to the fact that, unlike in the
dense case, subgraphs counts no longer characterise graph convergence.
We compare our random sparse graph sequence with an ``expected''
(deterministic) sparse graph via a concentration inequality. We then
use the Stein-Chen method to compare this deterministic sparse graph
to a sequence of graphs which are close to the graphon of the
underlying network.

The rest of the article is organised as follows. In Section~\ref{sgc}
we provide a brief introduction to sparse graph convergence. In
Section~\ref{mm} we describe our model and state our main result
(Theorem~\ref{mt1}). We present the proof of the main result in
Section~\ref{pmt1}. We then conclude with some remarks in a final
discussion section on Respondent Driven Sampling and Dense graph sequences.

\paragraph{Acknowledgements:} Adrian R\"ollin was supported by NUS Research Grand R-155-000-167-112. Siva Athreya was supported by CPDA grant from the Indian Statistical Institute and an ISF-UGC project grant.

\section{Sparse graph convergence} \label{sgc}

\noindent This section is a very brief introduction to sparse graph
convergence. The convergence of sparse graphs was initiated by
\cite{Bollobas2009} and then the~$L_p$ theory was established in
\cite{Borgs2014a} and \cite{Borgs2014b}. We present the minimal amount of
material necessary to formulate and prove our main result. We first
define weighted graphs, followed by definition of graphon and conclude
with a brief discussion on a convergence result.

\paragraph{Weighted graphs.} 

Consider a graph~$G$, given by its set of vertices~$V(G)$ and set of edges~$E(G)$. A (edge-)weighted graph~$G$ is simply a graph which has, in addition, a weight function~$\beta(G) = (\beta_{ij}(G))_{i,j\in V(G)}$, where, for each~$\{i,j\}\in E(G)$, we interpret the value~$\beta_{ij}(G)$ as the weight of that edge. By making the convention that~$\beta_{ij}(G)=0$ whenever there is no edge between vertices~$i$ and~$j$, the information about~$E(G)$ is contained in~$\beta(G)$, so that any weighted graph is determined by~$V(G)$ and~$\beta(G)$. Moreover, any unweighted graph can be interpreted as a weighted graph by setting~$\beta_{ij}(G)=1$ whenever~$\{i,j\}\in E(G)$.

For any weighted graph~$G$ and any constant~$c\in \IR$, we shall define~$cG$ to be the weighted graph on the same set of vertices and edge weights~$\beta_{ij}(cG) = c\beta_{ij}(G)$.

\paragraph{Graphons.} 

A graphon is any symmetric, function~$\kappa:[0,1]^2\to \IR_+$ which is integrable; note that we restrict ourselves to non-negative graphons, whereas \cite{Borgs2014a} allow for more general graphons.  For any graphon~$\kappa$, the \emph{cut-norm of~$\kappa$} is defined as
\be{ \label{cut-norm} 
 	\norm{\kappa}_\square 
		\vcentcolon= \sup_{S,T \subseteq [0,1]} \bbbabs{\int_{S \times T} \kappa(x,y) dx dy},
} 
where the supremum is taken over Lebesgue-measurable subsets of~$[0,1]$. The~$L_1$-norm of~$\kappa$ is given by
\be{ \label{L1norm}
	\norm{\kappa}_1 
	\vcentcolon=  \int_{[0,1]\times[0,1]}\abs{\kappa(x,y)} dx dy.
}
For any two graphons~$\kappa_1$ and~$\kappa_2$, we let
\ben{ 
	d_\square(\kappa_1, \kappa_2) \vcentcolon= \norm{\kappa_1 - \kappa_2}_\square,
	\qquad
  d_1(\kappa_1, \kappa_2) \vcentcolon= \norm{\kappa_1 - \kappa_2}_1.
}
Since a Lebesgue measure preserving transformation of~$[0,1]$ will not change the norm of a graphon, it is customary to define the cut-metric~$\delta_\square$ on graphons by
\ben{
	\delta_\square(\kappa_1, \kappa_2) \vcentcolon= \inf_\sigma d_\square(\kappa_1^\sigma, \kappa_2),
}
where the infimum ranges over all measure-preserving bijections~$\sigma:[0,1]\to[0,1]$, and where the graphon~$\kappa^\sigma$ is defined as~$\kappa^\sigma(x,y) = \kappa(\sigma(x), \sigma(y))$.

Every weighted graph~$G$ is naturally associated with a graphon~$\kappa_G$ in the following way. First, divide the interval~$[0,1]$ into intervals~$I_1, \dots, I_{\abs{V(G)}}$ of lengths~$1/\abs{V(G)}$ for each~$i \in V(G)$. The function~$\kappa_G$ is then given the constant value~$\beta_{ij}(G)$ on~$I_i \times I_j$ for every~$i, j \in V(G)$. It is easily verified that~$\kappa_G$ is indeed a graphon.

Thus, even if~$G$ and~$G^\prime$ have different set of vertices, we can define their cut-distance through the cut-distance of their associated graphons; that is,
\be{
	\delta_{\square}(G, G^\prime) \vcentcolon= \delta_\square(\kappa_G,\kappa_{G^\prime}).
}
If two weighted graphs~$G$ and~$G^\prime$ have the same set of vertices~$V(G)$, then it is clear that we can express their cut-distance as
\be{
	d_{\square}(G, G^\prime)  
	= \max_{S,T \subseteq V(G)} \bbbabs{\sum_{i \in S, j \in T} \beta_{ij}(G) - \beta_{ij}(G')}.
}
Finally, if~$\kappa$ is a graphon and~$G$ is a weighted graph, then we will define
\ben{
	d_\square(G, \kappa) \vcentcolon=   d_\square{(\kappa_G, \kappa)},
	\qquad
	\delta_\square(G, \kappa) \vcentcolon=   \delta_\square{(\kappa_G, \kappa)}.
}

\paragraph{Convergence to graphon.} 

Let~$\kappa$ be a graphon with~$\norm{\kappa}_1 > 0$.  Let~$\rho_n >0$ satisfy~$\rho_n \to 0$ and~$n \rho_n \to \infty$ as~$n \to \infty$. Let the vertex set be given by~$[n] = \{1, 2, \dots,
n\}$. Let~$U_1, \dots, U_n$ be i.i.d.\ chosen uniformly in~$[0,1]$.

Define~$G_n\equiv G(n, \kappa,\rho_n)$ to be the graph defined by
connecting~$i$ and~$j$ with probability~$\min\{\rho_n \kappa(U_i,U_j),
1\}$. It is clear that~$G_n$ is a sparse graph sequence and in \cite[Theorem 2.14 and Corollary 2.15]{Borgs2014a} it is shown that,  with probability~$1$,
\be{
   d_\square\bklr{G_n/\rho_n, \kappa} \to 0 \qquad \text{and} \qquad  \delta_\square\bklr{G_n/\norm{{G_n}}_1, \kappa/
  \norm{\kappa}_1} \to 0
}
as~$n \to \infty$. In this article we generalise the above  result when the vertex labels come from a Markov Chain and the sparse graph is constructed after suitable clumping.

\section{Model and main results} \label{mm}

\subsection{Constructing a random graph from RDS}

\noindent We shall construct a sparse graph on~$[n]$ vertices driven by Respondent Driven Sampling (RDS).  We will sample~$N$ individuals, labelled $X_1,\dots,X_N$, where~$X_i\in[0,1]$. We note that the label space is chosen arbitrarily to be the unit interval only for the sake of mathematical convenience. After sampling, the individuals are clumped into~$n$ equally spaced bins, which we represent by the intervals $A_{n,i} = [(i-1)/n,i/n)$, where~$1\leq i\leq n$ (it is understood that~$A_{n,n}$ also includes the right-most point 1). We connect~$i$ and~$j$ if two successive individuals fall into bin~$A_i$ followed by bin~$A_j$ or vice-versa. We chose $N$ in such a way that the graph constructed is sparse and we establish an~$L_1$ limit for the same. We begin with a precise definition of the sampling scheme via a Markov chain.



\paragraph{Markov Chain representing RDS.} Let~$\kappa$ be a graphon. Let
$(\Omega, {\cal F}, \IP)$ be a probability space, on which we define a
Markov chain~$X=\{X_k\}_{k\geq 0}$ with transition probabilities given
by 
\be{ 
	\IP[X_{m+1}\in dy|X_m=x] =
  \frac{\kappa(x,y)}{\int_0^1\kappa(z,x)dz}dy.
}
Since~$\kappa$ is symmetric, the Markov chain is time-reversible with stationary distribution 
\be{ 
	\pi(dx) 
		= \frac{\int_0^1\kappa(x,u)du}{\int_0^1\int_0^1 \kappa(u,v)dudv}dx.  
}
We shall assume that~$X_0~\eqlaw \pi$, which means the chain is stationary. Then the
probability of seeing a transition from~$dx$ to~$dy$ is given by
\besn{\label{mcjoint}
	\IP[X_m\in dx,X_{m+1}\in dy] & =
  \pi(dx)\frac{\kappa(x,y)}{\int_0^1\kappa(x,z)dz}dy  =
  \frac{\kappa(x,y)}{\int_0^1\int_0^1 \kappa(u,v)dudv}dxdy.
}

\paragraph{Sparse Random Graph from RDS.} Let~$\kappa$ be a graphon. Let~$n \geq 1$ and~$N\equiv N(n)$. We will now construct a random graph~$G(n,N, X, \kappa)$ via the following steps:
\begin{itemize}
\item Let the vertex set be~$[n]\vcentcolon= \{1, 2, \ldots, n\}$.
\item Let~$X_1,\dots,X_N$ be a realisation of the stationary Markov Chain defined in the previous section up to time~$N$. 
\item Equi-partition the unit interval by the intervals~$A_{n,1},\dots,A_{n,n}$. For~$1 \leq i, j \leq n$ with~$i\neq j$, define  
\be{
	I_n(i,j) = \begin{cases} 
		1 & \begin{minipage}[c]{0.6\textwidth}if there exists $0\leq m < N$ such that either~$X_m\in A_{n,i}$ and~$X_{m+1}\in A_{n,j}$, or $X_m\in A_{n,j}$ and $X_{m+1}\in A_{n,i}$,\end{minipage}\\[2ex]
		0 & \text{otherwise.} \end{cases}
}
\item For~$1 \leq i,j \leq n$ with $i\neq j$, connect~$i$ and~$j$ if~$I_n(i,j) = 1$, and leave it unconnected otherwise.
\end{itemize}
If we choose $N(n)$ appropriately (i.e. $N(n) = o(n^{2})$) then the above random graph will be a sparse random graph sequence.

\subsection{Main Result}

\noindent Let $\kappa$ be a given graphon, and consider the sparse graph sequence~$G_{n}\equiv G(n,N,X,\kappa)$ defined as in the previous paragraph. We shall make the following assumptions on~$\kappa$ and~$N$.
\begin{description}
\item[\hskip2em Assumption (K1).] There are a constant~$\delta > 0$ and an integrable function~$\phi: [0,1] \to \IR_+$ such that
\besn{ \label{conditiononkernel}
  0< \delta\leq \frac{\kappa(x,y)}{\int_0^1\kappa(x,z)dz} 
  \leq \phi(y),\qquad 0\leq x, y\leq 1.
}
\item[\hskip2em Assumption (N1).] There are constants~$\alpha$ and~$\lambda$, where~$0 < \alpha \leq 1$ and $\lambda>0$, such that the sequence $N\equiv N(n)$ satisfies 
\besn{\label{nNlim} 
	\lim_{n \to \infty} \frac{N}{n^{1+\alpha}} = \lambda.} 
\end{description}
We are now ready to state the  main result.

\begin{theorem} \label{mt1} Under Assumption (K1) and Assumption (N1), and if $0 < \alpha < 1$,
\be{
 \lim_{n\to \infty}d_\square\bbbklr{\frac{n^2}{N}G_{n} \,\, ,\,\, \frac{\kappa}{\norm{\kappa}_1}} = 0
}  
almost surely with respect to~$\IP$.
\end{theorem}

\section{Proof of Theorem~\ref{mt1}} \label{pmt1}

\noindent To prove our result we will need to define two (deterministic and intermediate) weighted graphs. The first graph is an ``averaged'' version of~$G_n$, which we shall denote by~$\IE G_n$; it is the weighted graph on the vertices~$[n]$ with edge weights
\be{
	\beta_{ij}(\IE G_n) = \frac{1}{2}\IE I_n(i,j).
}
Denote by~$\frac{n^2}{N}\IE G_n$ be the weighted graph obtained by scaling the weights of~$\IE G_n$ by~$\frac{n^2}{N}$ (as described in Section~\ref{sgc}).
The second graph, denoted by~$H_n$, is the weighted graph on the vertices~$[n]$ with edge weights
\ben{ \label{hn}
	\beta_{ij}(H_n) = \frac{n^2}{2N}\bbbklr{1-\exp\bbklr{-\frac{2N}{n^2}\mu_n(i,j)}},
\quad\text{where }
	\mu_n(i,j) = n^2\int_{A_{n,i}}\int_{A_{n,j}}\frac{\kappa(x,y)}{\norm{\kappa}_1}dxdy.
}
For~$x,y\in[0,1]$, let~$i_n$ and~$j_n$ be such that~$x\in A_{n,i_n}$ and
$y\in A_{n,j_n}$ for all~$n \geq 1$. Observe that by the Lebesgue
density theorem, $\kappa(x,y)/\norm{\kappa}_1 = \lim_{n\toinf}\gamma_n(i_n,j_n)~$
almost everywhere on~$[0,1]^2$.

Our strategy will be to show that, for large~$n$, $\frac{n^2}{N}G_{n}$ is close to~$\frac{n^2}{N}\IE G_{n}$, followed by the fact that~$\frac{n^2}{N}\IE G_{n}$ is close to~$H_n$, and finally that~$H_n$ is close to~$\kappa/\norm{\kappa}_1$.

We start with the first lemma, which shows that the distance between~$\frac{n^2}{N}G_{n}$ and~$\frac{n^2}{N}\IE G_{n}$ goes to~$0$ almost surely with respect to~$\IP$. The key ingredient of the proof is a concentration inequality of \cite{Paulin2015}.

\begin{lemma} \label{r1} We have
\besn{ \label{r1e}
  \lim_{n \to \infty} 
  	d_\square\bbbklr{\frac{n^2}{N}G_{n}, \frac{n^2}{N}\IE G_{n}}  = 0 
	 \qquad\text{almost surely  w.r.t.\ } \IP.
}
\end{lemma}

\begin{proof}
Note that
\bes{
  &\dbox\bbbklr{\frac{n^2}{N}G_{n},\frac{n^2}{N}\IE G_{n}}\\
	&\qquad = \sup\limits_{\substack{S=\bigcup_{m=1}^k A_{n,i_m},\\T=\bigcup_{m=1}^l A_{n,j_m}}
} \bbbabs{\sum\limits_{i:i/n\in S}\sum\limits_{j:j/n\in T}\frac{n^2}{2N}\bklr{I_n(i,j)-\IE I_n(i,j)}\Vol(A_{n,i})\Vol(A_{n,j})} \\
	&\qquad =  \sup\limits_{\substack{S=\bigcup_{m=1}^k A_{n,i_m},\\T=\bigcup_{m=1}^l A_{n,j_m}}
} \bbbabs{f_{S,T}(X)-\IE f_{S,T}(X)},
}
where $f_{S,T}(X) = f_{S,T}(X_0,\dots,X_N) = \frac{1}{2N}\sum_{i:i/n\in S}\sum_{j:j/n\in T}I_n(i,j)$. As~$X$ is Harris recurrent by Assumption (K1), we obtain from \cite[Theorem 16.0.2]{Meyn2009} that the Markov chain has finite mixing time $t_{mix}$.  Let~$\eps >0$  be given. Now, changing one point in~$f(X)$ will change~$f$ by at most 2 edges; that is, $f$ is~$1/N$-Hamming-Lipschitz. Therefore, by \cite[Corollary~2.10]{Paulin2015},
\be{
	\IP\bkle{ \abs{f_{S,T}(X)-\IE f_{S,T}(X)} > \eps } \leq 2\exp\bklr{- NC\eps^2},
}	
where~$C$ is a constant that only depends on~$t_{mix}$.   Using the union bound,
\bes{
	  \IP\bkle{ \sup_{\substack{S=\bigcup_{m=1}^k A_{n,i_m},\\T=\bigcup_{m=1}^l A_{n,j_m}}}\abs{f_{S,T}(X)-\IE f_{S,T}(X)} > \eps } 
	  		& \leq 2^{2n+1}\exp\bklr{- NC\eps^2} \\
		& = \exp\bklr{- NC\eps^2 + (2n+1)\log 2}.
} 
By \eq{nNlim} and Borel-Cantelli, the claim follows.
\end{proof}

Our second lemma shows that the distance between~$\frac{n^2}{N}\IE
G_{n}$ and~$H_n$ goes to~$0$. The key ingredient of the proof is an
application of the Stein-Chen method.

\begin{lemma}\label{r2} We have
\besn{\label{r2eq} 
	\lim_{n \to \infty}d_\square\bbbklr{\frac{n^2}{N}\IE G_{n},H_n}=0.
}
\end{lemma}
\begin{proof}Let
\be{
 E_n(i,j) = \sum_{m=1}^{N} I_{m},
}
where~$I_m = \I[(X_{m-1},X_m) \in \{A_{n,ij},A_{n,ji}\}]$ with~$A_{n,ij}=A_{n,i}\times A_{n,j}$, and note that
\be{
	\IE E_n(i,j) = 2N\int_{A_{n,i}}\int_{A_{n,j}}\frac{\kappa(x,y)}{\norm{\kappa}_1}dxdy = \frac{2N}{n^2}\mu_n(i,j).
}
Clearly, $I_n(i,j)=\I[E_n(i,j)>0]$. 
Now,
\besn{\label{tsc}
  &d_\square\bbklr{\frac{n^2}{N}\IE G_{n},H_n}\leq d_1\bbklr{\frac{n^2}{N}\IE G_{n},H_n}\\
	&\qquad = \sum_{i=1}^n \sum_{j=1}^n\frac{n^2}{2N}\Vol(A_{n,i})\Vol(A_{n,j})\babs{\IE I_n(i,j) - \bklr{1-\exp\klr{-\IE E_n(i,j)}}} \\
	&\qquad = \frac{1}{2N} \sum_{i=1}^n \sum_{j=1}^n\babs{\IE I_n(i,j) - (1-\exp\bklr{-\IE E_n(i,j)})} \\
	&\qquad = \frac{1}{2N} \sum_{i=1}^n \sum_{j=1}^n \abs{\IP[E_n(i,j)=0]-\IP[Z_n(i,j)=0]},
}
where~$Z_n(i,j)\eqlaw \Po\bklr{\frac{2N}{n^2}\mu_n(i,j)}$.
Now, let~$E_n^s(i,j)$ be a random variable having the size-bias distribution of~$E_n(i,j)$. Then, the Stein-Chen method (see, for example, \cite[Theorem~1.B]{Barbour1992}) yields 
\be{
	\dtv\bklr{\law(E_n(i,j)),\Po\bklr{{\textstyle\frac{2N}{n^2}}\mu_n(i,j)}} 
	\leq 
	\IE\babs{E_n(i,j)-(E^s_n(i,j)-1)},
}
where~$\dtv$ denotes the total variation distance.
Note that~$\IE I_m = p(i,j)$ for all~$0\leq m\leq n$, hence~$\IE I_m = \IE I_{m'}$ for all~$1\leq m,m'\leq n$. Thus, we can use the standard way to construct the size-bias distribution (see for example \cite{Goldstein1996}). To this end, let~$M$ be a uniformly chosen index from~$1$ to~$N$, independent of all else. It is not
difficult to show that~$\law(E_n(i,j)|I_M=1)$ is the size-bias distribution
of~$E_n(i,j)$. We now construct~$E_n^s(i,j)$ on the same probability space as~$E_n(i,j)$ in the following way. Consider~$M$ as given, and consider a process~$X'=(X'_0,\dots,X'_N)$ with law
\be{ 
	\law(X') =\law\klr{X\given I_M=1}=
  \law\bklr{X\given(X_{M-1},X_M)\in A_{n,ij}\cup A_{n,ji}}
  .
}
Let~$I'_m = \I[(X'_{m-1},X'_m) \in A_{n,ij}\cup A_{n,ji}]$, and observe that
\be{
	\law((I'_m)_{0\leq m\leq n}) = \law((I_m)_{0\leq m\leq n}\given I_M=1).
}
Thus, 
\be{ 
	E_n^s(i,j) = \sum_{m=1}^{N} I'_{m}
}
has the size-bias distribution of~$E_n(i,j)$.
If~$I_M = 1$, we can couple the two processes~$X$ and~$X'$ perfectly. If~$I_M=0$, we couple the two processes as follows. Condition~\eq{conditiononkernel} implies that~$X$ is Harris-recurrent; that is,
\be{ 
	\IP[X_{m+1}\in dy|X_m=x] \geq \delta dy.
}  
Thus, it is possible to couple~$X$ and~$X'$ such that
\bm{ 
	\IP[X'_{M+k}\neq X_{M+k}|I_M=0] \\ 
	\qquad\shoveleft{=\int_0^1\int_0^1\IP[X'_{M+k}\neq
      X_{M+k}\given X_{M-1}=x,X_M=y]}\\
      \shoveright{\times\IP[X_{M-1} \in dx,X_M \in dy|I_M=0]} \\
   \qquad\shoveleft{\leq (1-\delta)^k,\hfil} 
} 
and, similarly,
\be{ 
	\IP[X'_{M-1-k}\neq X_{M-1-k}|I_M=0] 
	\leq (1-\delta)^k.
}
We can easily extend the processes~$X_m$ and~$X_m'$ so that~$I_m$ and~$I'_m$ are defined for all~$m\in\IZ$. Now, let~$G_1$ and~$G_2$ be geometric random variables with success probability~$\delta$ dominating the coupling time forward and backward in time from~$M$ and~$M-1$ respectively. Note that we can construct~$G_1$ and~$G_2$ such that~$(G_1,G_2)\indep X$ and~$(G_1,G_2)\indep X'$ (note, however that
$(G_1,G_2)\not\indep (X,X')$). Then, 
\bmn{ \label{isb}
  \IE\babs{E_n(i,j)-(E_n^s(i,j)-1)} \\
   \qquad\shoveleft 
   \leq \IE I_M + \IE\bbbklg{(1-I_M)\sum_{m\neq M} \abs{I_m-I'_m}}\\ 
  \qquad\shoveleft
  \leq \IE I_M + \IE \sum_{m=1}^{G_1} (I_{M-m}+I'_{M-m})+\IE
  \sum_{m=1}^{G_2} (I_{M+m}+I'_{M+m}) \\
  \qquad\shoveleft \leq \IE I_M
  + \IE \sum_{m=1}^{\infty} \I[G_1\geq
    m](I_{M-m}+I'_{M-m})+\IE \sum_{m=1}^{\infty}
    \I[G_2\geq m] (I_{M+m}+I'_{M+m})  \\ \qquad\shoveleft \leq p(i,j)
  + \frac{2p(i,j)}{\delta} + \sum_{m=1}^\infty \IP[I_{M-m}=1|I_M=1]
  (1-\delta)^{m-1} \\ \shoveright{+ \sum_{m=1}^\infty
    \IP[I_{M+m}=1|I_M=1] (1-\delta)^{m-1}}\\ \qquad\shoveleft \leq
  p(i,j)\bbklr{1 + \frac{2}{\delta}} + \sum_{m=1}^\infty
  \bklr{\IP[I_{-m}=1|I_0=1]+\IP[I_{m}=1|I_0=1]} (1-\delta)^{m-1}\hfil
}
Now,
\bes{
  \IP[I_{m}=1|I_0=1] 
  &= \IP[(X_{m-1},X_m)\in A_{n,ij}\cup
    A_{n,ji}|I_0=1] \\ 
   &\leq \IP[X_m\in A_{n,i}\cup A_{n,j}|I_0=1] \\
   & =
  \int_{x\in[0,1]}\int_{y\in A_{n,i}\cup
    A_{n,j}}\frac{\kappa(x,y)dy}{\int_0^1\kappa(z,x)dz}\IP[X_{m-1}\in dx|I_0=1]\\
    &\leq
  \int_{x\in[0,1]}\int_{y\in A_{n,i}\cup A_{n,j}}\phi(y)dy\IP[X_{m-1}\in
    dx|I_0=1]\\ 
    &= \int_{y\in A_{n,i}\cup A_{n,j}}\phi(y)dy
    \leq
  \int_{y\in A_{n,i}}\phi(y)dy+\int_{y\in A_{n,j}}\phi(y)dy.
}
Applying this bound to~\eq{isb}, we have, for each~$i$ and $j$,
\be{	\IE\babs{E_n(i,j)-(E_n^s(i,j)-1)}\leq
	\bbklr{1+\frac{2}{\delta}}p(i,j)+\frac{1}{\delta}\int_{y\in A_{n,i}}\phi(y)dy+\frac{1}{\delta}\int_{y\in A_{n,j}}\phi(y)dy.
}
In conjunction with~\eq{tsc} and interchanging summation with integration, we arrive at
\bes{
	&d_\square\bbbklr{\frac{n^2}{N}\IE G_{n},H_n} \\
	&\qquad \leq \frac{1}{2N} \sum_{i=1}^n \sum_{j=1}^n \bbbklr{\bbklr{1+\frac{2}{\delta}}p(i,j)+\frac{1}{\delta}\int_{y\in A_{n,i}}\phi(y)dy+\frac{1}{\delta}\int_{y\in A_{n,j}}\phi(y)dy} \\
	&\qquad = \frac{(1+2\delta)n}{2N} + \frac{n}{\delta N}\int_{0}^1 \phi(x)dx.
}
Using \eq{nNlim}, the claim follows. 
\end{proof}

Our third lemma shows that the distance between~$H_n$ and~$\kappa/\norm{\kappa}_1$ goes to~$0$. The proof is a basic exercise in real analysis. 

\begin{lemma} \label{r3} We have
\besn{	
	\lim_{n \to \infty}  d_\square\bbbklr{H_n, \frac{\kappa}{\norm{\kappa}_1}} = 0.
}
\end{lemma}
\begin{proof}  To simplify writing, we introduce the notation~$\-\kappa \vcentcolon= \kappa/\norm{\kappa}_1$. 
Recall that~$H_n$ is the weighted graph on the vertices~$[n]$ with edge weights as in \eq{hn}. Define the graphon~$\^\kappa_n$ by 
\ben{ \label{100}
	\^\kappa_n(x,y) = n^2\int_{A_{n,i}}\int_{A_{n,j}} \-\kappa(u,v)dudv
	\qquad\text{for~$x,y\in A_{n,i} \times A_{n,j}$. }
}
Let $g_n$ be the graphon associated with the graph~$H_n$, which is given by 
\be{
	g_n(x,y)=  \frac{n^2}{2N}\bbklr{1-\exp\bbklr{-\frac{2N}{n^2} \^\kappa_n(x,y)}}.
}
Now, 
\ben{\label{fsr3}
	d_\square\bklr{H_n,\-\kappa} 
	\,\leq\, d_1\bklr{H_n,\-\kappa}
	 \,=\, \norm{g_n - \-\kappa}_1 
	 \,\leq\, \norm{g_n - \^\kappa_n}_1 + \norm{\^\kappa_n - \-\kappa}_1.
}
By \cite[Lemma~5.6]{Borgs2014b}, 
\ben{\label{fa0}
	\norm{\^\kappa_n-\-{\kappa}}_1\to 0.
}
Note that, by Taylor's approximation, $0\leq x - (1-e^{-x}) \leq \min\klg{x,x^2}$ for~$x>0$. Hence, we have for any~$x,y \in \IR$ that
\bes{
	\abs{g_n(x,y) -\^\kappa_n(x,y)} 
	& =  \frac{n^2}{2N} \bbabs{\bbklr{1-\exp\bbklr{-\frac{2N}{n^2}\^\kappa_n(x,y)}}-\frac{2N}{n^2}\^\kappa_n(x,y)}\\
&\leq \frac{n^2}{2N}\min\bbbklg{\frac{2N}{n^2}\^\kappa_n(x,y),\bbklr{\frac{2N}{n^2}\^\kappa_n(x,y)}^2}\\
  &= \min\bbbklg{1,\frac{2N}{n^2}\^\kappa_n(x,y)}\^\kappa_n(x,y).
} 
Let~$\tau >0$ (to be chosen later). For any graphon~$h$, let~$h\wedge\tau$ be the graphon defined as $(h\wedge \tau)(x,y) \vcentcolon= h(x,y)\wedge \tau$ and let the graphon $(h)_n$ be defined analogously to~\eq{100}.  Now, 
\besn{\label{sgaf1}
\norm{g_n-\^\kappa_n}_1 & \leq \int_0^1\int_0^1 \min\bbbklg{1,\frac{2N}{n^2}\^\kappa_n(x,y)}\^\kappa_n(x,y)dxdy \\
& \leq  \int_0^1\int_0^1 \min\bbbklg{1,\frac{2N}{n^2}\^\kappa_n(x,y)}(\^\kappa\wedge\tau)_n(x,y)dxdy + \norm{\^\kappa_n - (\^\kappa\wedge \tau)_n}_1.
}
By the contraction property,
\ben{\label{cpro}
	\norm{\^\kappa_n - (\^\kappa\wedge\tau)_n}_1 
	\,=\, \norm{(\^\kappa - \^\kappa\wedge\tau)_n}_1 
	\,\leq\, \norm{\-\kappa - \-\kappa\wedge\tau}_1.
}
Let~$\eps >0$. Then there exists~$\tau >0$ such that 
\ben{\label{kma1}
	\norm{\-\kappa - \-\kappa\wedge\tau}_1 <  \eps.
} 
For this choice of~$\tau$, as~$\min\bklg{1,\frac{2N}{n^2}\^\kappa_n(x,y)}(\^\kappa\wedge\tau)_n(x,y)$ converges to zero pointwise and is bounded by~$\tau$, we can use dominated convergence to conclude that there exists~$n_0 >0$ such that 
\besn{\label{mg0}
  \int_0^1\int_0^1 \min\bbbklg{1,\frac{2N}{n^2}\^\kappa_n(x,y)}(\^\kappa\wedge \tau)_n(x,y)dxdy & < \eps
}
for all~$n \geq n_0$. Therefore, applying \eq{cpro}--\eq{mg0} to \eq{sgaf1}, we have that 
\be{
	\norm{g_n - \^\kappa_n} < 2 \eps
	\qquad
	\text{for all $n\geq n_0$.}
}
As~$\eps >0$ was arbitrary, we conclude that
\ben{\label{ga0}
	\norm{g_n-\^\kappa_n}_1 \to 0.
} 
From \eq{fsr3}, \eq{fa0}, and \eq{ga0}  the claim now follows.
\end{proof}

We are now ready to prove the main result. It follows immediately from
the triangle inequality and the above three lemmas.

\begin{proof}[Proof of Theorem~\ref{mt1}] As indicated above using the triangle inequality, we have 
\be{
	d_\square\bbbklr{\frac{n^2}{N}G_{n},\frac{\kappa}{\norm{\kappa}_1}} 
	\leq 
	d_\square\bbbklr{\frac{n^2}{N}G_{n},\frac{n^2}{N}\IE G_{n}}
	+ d_\square\bbbklr{\frac{n^2}{N}\IE G_{n},H_n} 
	+ d_\square\bbbklr{H_n,\frac{\kappa}{\norm{\kappa}_1}}.
}
Application of  Lemma~\ref{r1}, Lemma~\ref{r2} and Lemma~\ref{r3} completes the proof.
\end{proof}

\section{Discussion}
We conclude this note, with some remarks on dense graph sequences and Respondent Driven Sampling.

\paragraph{Dense Graph Sequence.} We have chosen~$0 < \alpha < 1$ so as to ensure that the graph sequence was sparse. If~$\alpha =1$, then we obtain a dense graph sequence. In this case as well, the convergence in the cut-metric would hold but to a ``Poissonised''~$\kappa$ in the following sense. 

\begin{proposition} \label{mt2} Under Assumption (K1) and Assumption (N1), and if~$\alpha=1$,
\be{
 \lim_{n\to \infty}d_\square\bbbklr{\frac{n^2}{N}G_{n} ,\^\kappa} = 0
}  
almost surely with respect to~$\IP$, where the graphon $\^\kappa$ is given by 
\be{
	\^\kappa(x,y) = \lambda^{-1}\bklr{1-e^{-\lambda\-\kappa(x,y)}}.
}
\end{proposition}
\begin{proof}
  The proof follows the same way as the proof of Theorem~\ref{mt1}. So we provide a sketch.

  Lemma~\ref{r1} and Lemma~\ref{r2} hold for~$\alpha =1$ case as
  well. Instead of
  Lemma~\ref{r3}, we have to show \bes{ \lim_{n \to \infty}
    d_\square(H_n, \^\kappa) = 0.  } Define the graphon $f_n$
  as~$f_n(x,y)=\lambda^{-1}\bklr{1-e^{-\lambda\^\kappa_n(x,y)}}$. Now,
  \ben{\label{fsr3b} d_1\bklr{H_n,\^\kappa} \,=\, \norm{ g_n -
      \^\kappa}_1 \,\leq\, \norm{g_n - f_n}_1 + \norm{f_n -
      \^\kappa}_1.  } Recall that~$\abs{e^{-z} - e^{-w}} \leq
  \abs{z-w}$ for all~$z,w>0$.  So, for any~$x,y \in [0,1]$, \be{
    \abs{f_n(x,y) -\^\kappa(x,y)} = \abs{e^{-\lambda \^\kappa_n(x,y)}
      - e^{-\lambda \bar{\kappa}(x,y)}} \leq \lambda
    \abs{\^\kappa_n(x,y) - \bar{\kappa}(x,y)}.  } By
  \cite[Lemma~5.6]{Borgs2014b}, $\norm{\^\kappa_n-\bar{\kappa}}_1\to
  0$. Hence, using the above this readily implies \ben{\label{fa1}
    \norm{f_n-\^{\kappa}}_1\to 0.  } Note that for~$b \geq 0$ and $x
  >0$, \be{ \abs{xe^{-b/x} - ye^{-b/y}} \leq \abs{ x - y }.  } So, for
  any~$x,y \in \IR$, we have \bes{ \babs{ g_n(x,y) -f_n(x,y)} & =
    \bbbabs{\frac{n^2}{2N}\bbklr{1-\exp\bbklr{-\frac{2N}{n^2}\^\kappa_n(x,y)}}-
      \frac{1}\lambda\bbklr{1-\exp\bklr{-\lambda\^\kappa_n(x,y)}}}\\ &\leq
    \bbbabs{\frac{n^2}{2N} - \frac{1}{\lambda} } +\bbbabs{
      \frac{n^2}{2N}\exp\bbklr{-\frac{2N}{n^2}\^\kappa_n(x,y)}-
      \frac{1}{\lambda}\exp\bklr{-\lambda\^\kappa_n(x,y)}}\\ &\leq
    2\bbabs{\frac{n^2}{2N} - \frac{1}{\lambda} }, } As~$\babs{n^2/(2N)
    - 1/\lambda} \to 0$, dominated convergence implies
  \ben{\label{ga1} \norm{g_n-f_n}_1\to 0.  } From this the result
  follows as in the proof of Theorem~\ref{mt1}
\end{proof}

We note that the Stein-Chen method plays a critical role in proof of Lemma~\ref{r2} when $\alpha =1$, as $\frac{N}{n^2} \to \lambda >0$; that is, the mean of the Poisson random variable does not converge to $0$, so that moment bounds would not suffice to prove Lemma~\ref{r2}.

\paragraph{Respondent Driven Sampling (RDS).} One common approach in RDS to correct for bias towards high degrees, is to ask participants of
the study to estimate their own degree and then weigh the participants
by the inverse of their reported degree. This procedure is known as
\emph{multiplicity sampling}, and was first used in the context of RDS
by \cite{Rothbart1982}. What Theorem~\ref{mt1} implies in essence is that one could also clump participants together according to general characteristics (such as age, gender, etc.). If the degree of the participants is captured by these characteristics, the bias towards participants with high degrees would disappear.

It was argued by \cite{Heckathorn2007} that multiplicity sampling
cannot in general correct for the bias towards nodes with
high degree due to possible \emph{differential recruitment}, which
means that some groups of participants are systematically able to
recruit more people than others. Other methods of estimations,
including the original estimators of \cite{Heckathorn1997} as well as the clumping procedure proposed in this article, are
equally susceptible to differential recruitment bias.

The mathematical reason behind this bias is that the stationary
distribution of a one-referral Markov process on a set of types, which
is the commonly used mathematical tool to derive RDS estimators, can
be different from the stationary distribution of a multi-type
branching process with the same transition probabilities if the
average number of offspring depends on the types. This was
described precisely by \cite{Athreya2016}, where the two models, a
one-referral Markov chain and Poisson-offspring branching process,
show substantially different over-sampling of high-degree vertices in
the network. In the one-referral Markov chain case, the over-sampling
is exactly proportional to the degree, but in the case of a Poisson
number of referrals, it is proportional to a quantity that is harder
to calculate (the eigenfunction of the mean replacement measure of the
branching process). In practice, differential recruitment bias is
typically reduced by limiting the number of referrals, traditionally
to no more than three.

\cite{Heckathorn2007} also proposes a method, called estimation through \emph{dual-components}, which is supposed to take differential recruitment into account. This is the default method used in the widely-used statistical software \emph{RDSAT} (see \cite{Volz2012}). The basic idea is to estimate the transition probabilities governing the referrals, calculate the proportion of different types one would expect to see under absence of both bias due to different degrees and bias due to differential recruitment, compare with the actual observed proportions, and then to work backwards to find the true proportions in the population. However, the theoretical justifications in \cite{Heckathorn2007} for the details of the procedure are somewhat opaque.

\paragraph{Open Problems.} We conclude the article with a couple of questions that can be explored.

\begin{enumerate}[label=(\arabic*)]
\item In \cite{Athreya2016} a rigourous framework was set up to handle convergence in dense graph limits. For dense graphs, the theory of graphons (whose range is $[0,1]$) was used to establish the convergence.  Graphons in dense graph setting characterise the limit via convergence of subgraph counts. This aspect applies under several equivalent metrics. One should be able to establish the RDS models used in \cite{Athreya2016} to prove convergence in the $L_1$ metric as in this article. The approach could be one as laid out in proof of \cite[Theorem 2.14]{Borgs2014a}.
\item As already mentioned before, in practice, an RDS sample comes typically in the form of a tree, rather than a single chain, and hence, a multi-type branching process, where the types could represent characteristics such as gender, age etc., would constitute a more realistic mathematical model. The stationary distribution of such a branching process is difficult to solve analytically in general, but under additional assumptions, such as considering only finitely many types, a numerical approach would definitely be feasible. In this light, it seems that a statistical theory based on branching process theory, rather than Markov chain theory, could put the framework of \emph{dual-components} from \cite{Heckathorn2007} onto solid ground or even improve on it.
\end{enumerate}


\end{document}